\documentclass{article}
\usepackage{amssymb, amsmath, amsthm, bbm, multirow, enumitem}
\usepackage{mathtools}
\usepackage[top=3cm, bottom=3cm, left=3cm, right=3cm]{geometry}
\usepackage{graphicx} 
\usepackage{hyperref}
\usepackage{xcolor}
\hypersetup{
  colorlinks   = true, 
  urlcolor     = blue, 
  linkcolor    = red, 
  citecolor   = blue
}
\usepackage{color}

\usepackage[round]{natbib}
\usepackage{comment}

\newtheorem{theorem}{Theorem}[section]
\newtheorem{lemma}{Lemma}[section]
\newtheorem{proposition}{Proposition}[section]

\newtheorem*{corollary*}{Corollary}
\newtheorem*{remark}{Remark}

\theoremstyle{definition}
\newtheorem{definition}{Definition}[section]

\newcommand{\var}{\mathrm{Var}}

\newcommand{\integers}{\mathbb{Z}}
\newcommand{\reals}{\mathbb{R}}
\newcommand{\nat}{\mathbb{N}}

\newcommand{\dd}{{\rm d}}
\newcommand{\Falg}{\mathcal{F}}
\newcommand{\eps}{\varepsilon}
\newcommand{\tr}{\top}

\newcommand{\prob}{\mathbb{P}}
\newcommand{\E}{\mathbb{E}}

\newcommand{\dyadic}{\mathcal{D}}

\newcommand{\xspace}{\mathcal{X}}
\newcommand{\yspace}{\mathcal{Y}}
\newcommand{\diff}{\mathrm{d}}

\newcommand{\given}{\,|\,}
\newcommand{\interval}{\mathcal{I}}

\title{Asymptotic properties of adaptive designs\\ through differentiability in quadratic mean}
\author{Dennis Christensen$^{1,2}$ \and Emil Aas Stoltenberg$^{1,3}$ \and Nils Lid Hjort$^1$}
\date{\small {$^1$Department of Mathematics, University of Oslo \\
$^2$Norwegian Defence Research Establishment (FFI) \\
$^3$Department of Data Science, BI Norwegian Business School}}

\newcommand{\vbb}[1]{\lVert#1\rVert}
\newcommand{\half}{\tfrac12}

\newcommand{\quart}{\tfrac14}

\defcitealias{stanag4487}{NATO, 2009a}
\defcitealias{stanag4488}{NATO, 2009b}
\defcitealias{stanag4489}{NATO, 1999}
\defcitealias{mil-std-331D}{U.S.~Department of Defense, 2017}
\defcitealias{burke2017binary}{Burke and Truett, 2017}
\defcitealias{baker2021overview}{Baker, 2021}

\providecommand{\keywords}[1]
{
  \small	
  \textbf{\textit{Keywords---}} #1
}

\begin{document}

\maketitle

\begin{abstract}
    There exist multiple regression applications in engineering, industry and medicine where the outcomes follow an adaptive experimental design in which the next measurement depends on the previous observations, so that the observations are not conditionally independent given the covariates. In the existing literature on such adaptive designs, results asserting asymptotic normality of the maximum likelihood estimator require regularity conditions involving the second or third derivatives of the log-likelihood. Here we instead extend the theory of differentiability in quadratic mean (DQM) to the setting of adaptive designs, which requires strictly fewer regularity assumptions than the classical theory. In doing so, we discover a new DQM assumption, which we call summable differentiability in quadratic mean (S-DQM). As applications, we first verify asymptotic normality for two classical adaptive designs, namely the Bruceton `up-and-down' design and the Robbins--Monro design. Next, we consider a more complicated problem, namely a Markovian version of the Langlie design.
    
\end{abstract}


\section{Introduction}\label{sec:introduction}
Consider a regression model $f_\theta(y_i\given x_i)$, parametrised by $\theta$, which aims to predict the outcome $y_i\in\mathcal{Y}$ from covariates $x_i\in\mathcal{X}$, for $i=1, \dots, n$.  There are numerous applications in industry, medicine and toxicology where the covariates follow an adaptive design of the form
\begin{equation}\label{eq::covariate_seq_general}
    X_{i} = h(\{(X_j,Y_j)\}_{1 \leq j \leq i-1}, U_{i-1}),
\end{equation}
for $i=2, 3, \dots$, where $h$ is some user-defined function. Here, the noise variables $U_i$ are assumed to be independent, and also independent of the sequence $\{(X_i, Y_i)\}_{i\geq1}$. In many applications, the noise variables are not present and so the function $h$ depends only deterministically on the previous observations. In any case, the distribution of the next input $X_{i}$ depends on $(X_1, Y_1),\ldots,(X_{i-1},Y_{i-1})$, losing conditional independence between the $Y_i$, given the $X_i$.

Adaptive designs of the form \eqref{eq::covariate_seq_general} arise naturally in binary regression problems when estimating the sensitivity of a material, by which we mean the susceptibility of the material to break, fracture, or explode under various levels of physical stimulus. For example, we might want to estimate the impact sensitivity of an energetic material, which is done by repeatedly dropping a weight of known mass from different (log) heights $X_1, \dots, X_n$ onto samples of the material in question, and observing whether each impact causes an explosion $(Y_i = 1)$ or not $(Y_i = 0)$ \citep{meyer2007explosives, christensen2024estimating}. The probability of observing an explosion is modelled using binary regression,
\begin{equation}\label{eq:binarymodel}
    \prob(Y_i = 1 \given X_i) = H(\theta^\top Z_i) = H(\alpha + \beta X_i),
\end{equation}
where $\theta = (\alpha,\beta)^\top $, $Z_i = (1, X_i)^\top$ and $H:\reals\to[0,1]$ is a fixed cumulative distribution function, typically the logistic sigmoid or probit function. In other related fields such as bioassay, multiple tests are typically performed together at each dosage level, with binomial outcomes. Therefore, the dosage levels are typically not readjusted in between every single measurement.

The dependence imposed by~\eqref{eq::covariate_seq_general} complicates the large-sample theory for the maximum likelihood estimator $\widehat\theta_n$ of $\theta$. Indeed, the likelihood theory cannot rely on conditional independence of the outcomes, nor can it rely on the data being independent and identically distributed (i.i.d.). Thus, asymptotic normality of $\widehat\theta_n$ must therefore be explicitly verified for tasks such as constructing confidence intervals and hypothesis testing. The former is of particular interest, as it has been demonstrated in multiple simulation studies that confidence intervals constructed via Fieller's theorem outperform the traditional approach from the delta method for binary regression \citep{abdelbasit1983experimental, cox1990fiellers, sitter1993accuracy, faraggi2003confidence, christensen2023improved}. See also \citet[Section 4.6]{schweder2016confidence} for an interpretation of Fieller's theorem through confidence distributions.

The asymptotic study of adaptive designs is not new, and many results have been achieved by extending the standard machinery for the i.i.d.~case to the adaptive case \citep{tsutakawa1967asymptotic, fahrmeir1987asymptotic, rosenberger1997asymptotic}. These results, however, require regularity conditions involving the existence of two or three derivatives of the model $\theta \mapsto f_{\theta}(y \given x)$, as well as domination assumptions allowing us to exchange differentiation and integration. In contrast, the present paper instead imposes a certain differentiability in quadratic mean (DQM) assumption on the the conditional density $f_{\theta}(y \given x)$, which we call {\it summable differentiability in quadratic mean} (S-DQM). In imposing this condition, we circumvent the traditional assumptions on the second or third derivatives of the conditonal density. In fact, even the first derivative of $f_{\theta}(y \given x)$ need not exist for all $(x,y)$. Our theory mirrors the standard quadratic mean theory for i.i.d.~data \citep{vandervaart1998asymptotic, pollard1997another}, but extends that theory to dependent regression data of the form described above, with adaptive designs and binary outcomes as a prime application. This S-DQM theory is presented in Section~\ref{sec:2}.

A factor that significantly complicates the analysis of adaptive designs, both using classical methods or S-DQM, is that one needs to study the properties of the covariate sequence $(X_n)$ on a case-to-case basis to ensure that the observed information has a limit in probability. This can be achieved in a number of ways, depending on the properties of the covariate sequence. For instance, it will be shown that if $X_n$ itself converges in probability, then the observed information also converges. Another simplifying scenario is when $(X_n)$ is a well-behaved Markov chain, in which case one can appeal to ergodic theorems. In Section~\ref{sec:illustrations}, we illustrate the use of both these strategies to verify asymptotic normality of the maximum likelihood estimator for the Bruceton and Robbins--Monro designs, both of which serve as classical adaptive designs in sensitivity testing. In Section~\ref{sec:langlie}, we consider a more complicated example, namely a Markovian version of the Langlie design. Interestingly, even though the design itself yields a discrete state space Markov chain $(X_n)$, we require a detour into continuous state space theory to achieve an irreducible chain to which an ergodic theorem can be applied.

\section{Local asymptotic normality for adaptive designs}\label{sec:2}
In this section we study local asymptotic normality expansions of likelihood ratio process in regression settings where the conditional densities of the outcomes are {\it summable differentiable in quadratic mean}, a notion we define below.
\subsection{Summable differentiability in quadratic mean (S-DQM)}
Throughout, $(X_1,Y_1),(X_2,Y_2),\ldots$ are regression data such that $Y_i$ given $X_i$ has conditional density $f_{\theta_0}(y \given x)$ for all $i$, where $f_{\theta_0}$ belongs to some family $\{ f_{\theta} \colon \theta \in \Theta\}$ of conditional densities, and the true value $\theta_0$ is an inner point of $\Theta \subset \mathbb{R}^p$. Define $\ell_n(\theta) = \sum_{i=1}^n \log f_{\theta}(Y_i \given X_i)$, which we refer to as the log-likelihood; and the log-likelihood ratio process $A_n(h) = \ell_n(\theta_0 + h/\sqrt{n}) - \ell_n(\theta_0)$. Introduce the conditional scores $u_{\theta}(y \given x) = \partial \log f_{\theta}(y \given x)/\partial \theta$, and write $U_{n,j} = n^{-1/2}\sum_{i=1}^j u_{\theta_0}(Y_i \given X_i)$ for $j = 1,\ldots,n$, for the normalised score function evaluated in $\theta_0$. Let $\mu$ be the measure with respect to which $y \mapsto f_{\theta}(y \given x)$ is a conditional density. A key quantity in what follows is 
\begin{equation}
D_{\theta,h}(x) = \int \left\{\sqrt{f_{\theta + h}(y \given x)} - \sqrt{f_{\theta}(y \given x)} 
- \half h^\top u_{\theta}(y \given x) \sqrt{f_{\theta}(y \given x)} \right\}^2\,\mathrm{d} \mu(y).
\notag
\end{equation}
A family $\{f_{\theta}(\cdot \given \cdot) \colon \theta \in \Theta\}$ of conditional densities is {\it differentiable in quadratic mean} (DQM) at $\theta$ if
\begin{equation}
\E\, D_{\theta,h}(X)  = o(\vbb{h}^2), \quad \text{as $h \to 0$},
\label{eq::dqm}
\end{equation}
This is the classical DQM condition, see for example \citet[Eq.~(7.1), p.~93]{vandervaart1998asymptotic}, presented in a manner that suits our purposes. 
The DQM condition in~\eqref{eq::dqm} well suited for i.i.d.~data, but is not strong enough to handle adaptive designs, nor more generally situations where the the distribution of $X_i$ changes with $i$. This motivates introducing the notion of a model being what we choose to call {\it summable differentiability in quadratic mean}, or S-DQM for short. 
\begin{definition}{\sc(S-DQM)}. Let $(X_1,Y_1),(X_2,Y_2)\ldots$ be such that for each $i$, $Y_i$ given $X_i$ has conditional density $f_{\theta}(y \given x)$ for some $\theta \in \Theta$. The family $\{f_{\theta}(\cdot \given \cdot) \colon \theta \in \Theta\}$ is \emph{summable differentiable in quadratic mean} at $\theta$ if $\sum_{i=1}^n \E\,D_{\theta,h/\sqrt{n}}(X_i ) = o(1)$ as $n$ tends to infinity.
\end{definition}
The S-DQM property has been studied for independent data, albeit not with that name, see, e.g., \citet[Proposition~A.8, p.~182]{vandervaart1988large} or \citet[Corollary~74.4, p.~382]{strasser1985mathematical}. Even though S-DQM refers to a model or family $\{f_{\theta}(\cdot\given \cdot) \colon \theta \in \Theta\}$ of conditional densities and a given sequence $(X_1,Y_1),(X_2,Y_2)\ldots$ of regression data, in the following we simply say that $f_{\theta}(y \given x)$ is S-DQM. Notice that if the data $(X_1,Y_1),(X_2,Y_2)\ldots$ are i.i.d., then $\sum_{i=1}^n \E\,D_{\theta,h/\sqrt{n}}(X_i ) = n \E\,D_{\theta,h/\sqrt{n}}(X_1 )$, and S-DQM is equivalent to DQM. Also, since the random variables $D_{\theta,h/\sqrt{n}}(X_1),D_{\theta,h\sqrt{n}}(X_2),\ldots$ are non-negative, S-DQM implies that $\sum_{i=1}^n D_{\theta,h/\sqrt{n}}(X_i) = o_p(1)$. It is this convergence in probability to zero that will be used in the proof of Proposition~\ref{proposition:Anh_approxrep} below. Before we turn to that proposition we must introduce the process  
\begin{equation}
\langle U_{n,\cdot},U_{n,\cdot} \rangle_j = \frac1n\sum_{i=1}^j \E_{\theta_0}[u_{\theta_0}(Y_i \given X_i)u_{\theta_0}(Y_i \given X_i)^\top\given \mathcal{F}_{i-1}],\quad j = 1,\ldots,n.
\label{eq:qv}
\end{equation}
Here, $(\mathcal{F}_i)_{i \geq 1}$ is the filtration defined by $\mathcal{F}_{i} = \sigma((X_j,Y_j, U_j)_{1 \leq j \leq i})$. As the notation suggests, the process $\langle U_{n,\cdot},U_{n,\cdot} \rangle_j$ will be the predictable quadratic variation of $U_{n,j}$, but we have not yet established that $U_{n,j}$ is a square integrable martingale. The fact that it is follows from the conditions of the upcoming proposition. This proposition extends a well known `DQM implies local asymptotic normality' result for i.i.d.~data (see for example \citet[Theorem~7.2, p.~94]{vandervaart1998asymptotic} or \citet[Theorem~2, p.~577]{lecam1986}, to dependent data of the type we are studying in this paper.  
\begin{proposition}\label{proposition:Anh_approxrep} Suppose that the conditional density $f_{\theta}(y \given x)$ is S-DQM at $\theta_0$; that $i \mapsto \E_{\theta_0}\,\vbb{u_{\theta_0}(Y_i \given X_i)}^4$ is uniformly bounded; and that $\langle U_{n,\cdot},U_{n,\cdot}\rangle_n$ is bounded in probability. Then $(U_{n,j})_{1\leq j \leq, n \geq 1}$ is an array of square integrable martingales, and $A_n(h) = h^\top U_{n,n} - \half h^\top\langle U_{n,\cdot},U_{n,\cdot} \rangle_n h + o_p(1)$ for each $h$, as $n$ tends to infinity.
\end{proposition}
\begin{proof} See the Appendix.
\end{proof}
A couple of remarks are in place. First, that this theorem requires S-DQM, and not merely DQM, is the price we pay for allowing for dependence in the regression data. As mentioned above, S-DQM reduces to DQM when the data are i.i.d. The Lyapunov type condition on the fourth moment of the scores ensures that the difference $n^{-1}\sum_{i=1}^n u_{\theta_0}(Y_i \mid X_i)u_{\theta_0}(Y_i \mid X_i)^{\tr}-\langle U_{n,\cdot},U_{n,\cdot}\rangle_n$ converges in probability to zero. The need for a fourth moment here stems from the fact that we use Lenglart{'}s inequality to prove this convergence (see e.g.~\citet[Theorem~VII.3.4, p.~496]{shiryaev1996probability}). Contrast this with the i.i.d.~case, where the difference above equals $n^{-1}\sum_{i=1}^n u_{\theta_0}(Y_i \mid X_i)u_{\theta_0}(Y_i \mid X_i)^{\tr}- \E_{\theta_0}\,u_{\theta_0}(Y_1 \given X_1)u_{\theta_0}(Y_1 \given X_1)^{\tr} $, which is $o_p(1)$ by the law of large numbers provided the Fisher information $\E_{\theta_0}\,u_{\theta_0}(Y_1 \given X_1)u_{\theta_0}(Y_1 \given X_1)^{\tr}$ exists. That is, in the i.i.d.~case, only the second moment of the score needs to exist, thus the fourth moment condition also seems to be a price we pay for the dependence. The condition we impose on $\langle U_{n,\cdot},U_{n,\cdot}\rangle_n$ being $O_p(1)$ reduces to existence of the Fisher information when the data are i.i.d. In the i.i.d.~case, the existence of the Fisher information is a consequence of DQM, but a similar implication does not seems to hold for S-DQM. This discussion illustrates that the conditions impose in Proposition~\ref{proposition:Anh_approxrep} are natural extensions of the conditions imposed in the i.i.d.~version of this proposition (see the references cited before the proposition). 

\subsection{Techniques for establishing S-DQM}
We now turn to the problem of establishing that a model is S-DQM, and demonstrate a few techniques for doing so. First, introduce the matrix valued function
\begin{equation}
J_{\theta}(x) = \int u_{\theta}(y \given x)u_{\theta}(y \given x)^{\top} f_{\theta}(y \given x) \,\dd\mu(y). 
\label{eq:jtheta}
\end{equation}
The predictable quadratic variation can then be written $\langle U_{n,\cdot},U_{n,\cdot}\rangle_j = n^{-1}\sum_{i=1}^j J_{\theta_0}(X_i)$. Let $X$ be a covariate with distribution $P_{\theta}^X$. A classical lemma (see e.g.,~\citet[Lemma~7.6, p.~95]{vandervaart1998asymptotic}), with notation adapted to our setting, says that if $\theta \mapsto \sqrt{f_{\theta}(y \given x)}$ is continuously differentiable for all $(x,y)$ and $\int J_{\theta}(x) \,\dd P_{\theta}^X(x)$ is finite and continuous in $\theta$, then $\E_{\theta}\, D_{\theta,h/\sqrt{n}}(X) = o(1/n)$ for each $h$ as $n \to \infty$. Now, let $P_{\theta}^{X_1},P_{\theta}^{X_2},\ldots$ be the marginal distributions of the random variables $X_1,X_2,\ldots$. Then a crude bound on the expectation of the sum $\sum_{i=1}^n D_{h/\sqrt{n}}(X_i)$ is 
\begin{equation}
\E_{\theta}\,\sum_{i=1}^n D_{\theta,h/\sqrt{n}}(X_i)
\leq n \max_{i \leq n}\int D_{\theta,h/\sqrt{n}}(x)\,\dd P_{\theta}^{X_i}(x).
    \notag
\end{equation}
Using Lemma~7.6 from \citet{vandervaart1998asymptotic}, we might establish that $\int D_{\theta,h/\sqrt{n}}(x)\,\dd P_{\theta}^{X_i}(x) = o(1/n)$ for each $i$, but this is not enough to establish that neither the right nor the left hand side of the above display is $o(1)$. Assume, however, that that all the marginal distributions are dominated by a $\sigma$-finite measure $\nu$, and that that the densities $\dd P_{\theta}^{X_i}/\dd \nu$ are bounded by a function $g$, then $n \max_{i \leq n}\int D_{\theta,h/\sqrt{n}}(x)\,\dd P_{\theta}^{X_i}(x) \leq n \int D_{\theta,h/\sqrt{n}}(x) g(x)\,\dd\nu(x)$, and by the classical lemma $n \int D_{\theta,h/\sqrt{n}}(x) g(x)\,\dd\nu(x) = o(1)$, provided $\theta \mapsto \sqrt{f_{\theta}(y \given x)}$ is continuously differentiable for all $(x,y)$ and $\int J_{\theta}(x) g(x) \,\dd\nu(x)$ is finite and continuous in $\theta$. We summarise this discussion in a lemma, the complete proof of which can be found in the Appendix.

\begin{lemma}\label{lemma:lemma7.6ext2}Assume that there is a function $g$ and a $\sigma$-finite measure $\nu$ such that {\rm(i)} $ \theta \mapsto \sqrt{f_{\theta}(y \given x)}$ is continuously differentiable for all $(x,y)$; {\rm(ii)} $\theta \mapsto \int J_{\theta}(x) g(x)\,\dd\nu(x)$ is continuous in a neighbourhood of $\theta_0$; {\rm(iii)} $P_{\theta_0}^{X_{i}} \ll \nu$ for $i = 1,2,\ldots$; and {\rm(iv)} the densities $\dd P_{\theta_0}^{X_i}/\dd\nu,\, i = 1,2,\ldots$ are dominated by $g$. Then the model $\{f_{\theta}( \cdot \given \cdot) \colon \theta \in \Theta\}$ is S-DQM at $\theta_0$.
\end{lemma}
We remark that if the data $(X_1,Y_1),(X_2,Y_2),\ldots$ are i.i.d., in particular $P_{\theta}^{X_i} = P_{\theta}^{X_j}$ for all $i$ and $j$, then conditions (iii) and (iv) of Lemma~\ref{lemma:lemma7.6ext2} are redundant, and the conclusion of the lemma follows immediately from the classical lemma quoted above (and S-DQM is just DQM).

In addition to Lemma~\ref{lemma:lemma7.6ext2}, S-DQM can also be established for sufficiently smooth models on a compact domain, as the next lemma demonstrates.

\begin{lemma}\label{lemma:compact}
    Suppose that the variables $Y_i$ have finite support $\yspace = \{y_1, \dots, y_k\}$. Let $\xspace\subseteq\reals$ and $\Theta_0\subseteq\Theta$ be compact sets, and suppose that for each $y\in\yspace$, the map $(x, \theta) \mapsto  \sqrt{f_\theta(y\given x)}$ is continuous and has continuous first and second derivatives on $\xspace\times\Theta_0$. Then for any sequence $(X_i)$ of random variables with domain $\xspace$, the family $\{f_\theta(\cdot\given\cdot) : \theta\in\Theta_0\}$ is S-DQM at any point $\theta$ in the interior of $\Theta_0$.
\end{lemma}

\begin{proof}
    See the Appendix.
\end{proof}
In the case of binary regression, the $Y_i$ trivially have finite support, so this assumption is not too restrictive for our purposes. Also, for standard model choices for~\eqref{eq:binarymodel}, such as logistic or probit regression, the continuity assumptions in Lemma~\ref{lemma:compact} are satisfied provided $\beta > 0$. 

We conclude this section with an extension of Lemma~\ref{lemma:compact}, which will be useful for situations where the covariate sequence itself converges almost surely to a constant limit (like it does in the Robbins--Monro design). We state the result here, the proof of which can be found in the Appendix.

\begin{lemma}\label{lemma:conv-as}
    Suppose that the variables $Y_i$ have finite support $\yspace = \{y_1, \dots, y_k\}$, and that $X_i\to\gamma\in\xspace$ almost surely ($\xspace$ need not be compact). Let $\Theta_0\subseteq\Theta$ be a compact set, and suppose that for each $y\in\yspace$, the map $(x, \theta) \mapsto  \sqrt{f_\theta(y\given x)}$ is continuous and has continuous first and second derivatives on $\xspace\times\Theta_0$. Then for any point $\theta$ in the interior of $\Theta_0$, we have that $\sum_{i=1}^n D_{\theta, h/\sqrt{n}}(X_i) \to 0$ almost surely. In particular, $\sum_{i=1}^n D_{\theta, h/\sqrt{n}}(X_i) = o_p(1)$.
\end{lemma}

Note that the conclusion of Lemma~\ref{lemma:conv-as} does not necessarily imply S-DQM. However, as pointed out in the previous section, it is the convergence of $\sum_{i=1}^n D_{\theta, h/\sqrt{n}}(X_i)$ to zero in probability which is needed in Proposition~\ref{proposition:Anh_approxrep}, so Lemma~\ref{lemma:conv-as} is still useful to us.

\subsection{From S-DQM to asymptotic normality}
The next theorem will be used in establishing asymptotic normality of the maximum likelihood estimators in the adaptive designs we study in this paper. All quantities appearing in the theorem are as defined in above.
\begin{theorem}\label{thm:master_theorem} Suppose that conditions of Proposition~\ref{proposition:Anh_approxrep} are in force. In addition, assume that $A_n(h)$ is concave for each $n$; and that $\langle U_{n,\cdot},U_{n,\cdot}\rangle_n$ converges in probability to an invertible matrix $J$. Then $\sqrt{n}(\widehat{\theta}_n - \theta_0)$ converges in distribution to mean zero normally distributed random variable with covariance $J^{-1}$.
\end{theorem}
\begin{proof} From Proposition~\ref{proposition:Anh_approxrep} coupled with the convergence of $\langle U_{n,\cdot},U_{n,\cdot} \rangle_n$, we have that $A_n(h) = h^\top U_{n,n} - \half h^\top J h + o_p(1)$. By the concavity of $A_n(h)$, the basic corollary in \citet{hjort1993asymptotics} yields $\sqrt{n}(\widehat{\theta}_n - \theta_0) = J^{-1} U_{n,n} + o_p(1)$. The convergence in probability of $\langle U_{n,\cdot},U_{n,\cdot} \rangle_n$ and the Lyapunov condition ensures that $U_{n,n}$ converges in distribution to $\mathrm{N}(0,J)$ (see, for example, \citet[Theorem~3.3(a), p.~84]{helland1982central}). 
\end{proof}
Let us spend a bit of time on unpacking each of the conditions of Theorem \ref{thm:master_theorem} (and Proposition~\ref{proposition:Anh_approxrep}). Let $\theta = (\alpha,\beta)^{\tr}$, $z = (1,x)^{\tr}$ and $Z_i = (1,X_i)^{\tr}$. Regarding concavity, we note that in the binary regression setting, $h \mapsto A_n(h)$ is concave if $\theta \mapsto \log f_{\theta}(y \given x) = y \log H(z^{\top}\theta) + (1 - y)\log \{1 - H(z^{\top}\theta)\}$ is concave as the sum of concave functions is concave. Further, since affine functions preserve concavity, it suffices to check that $a \mapsto y \log H(a) + (1 - y)\log[1 - H(a)]$ is concave for $a$ ranging over the reals. This is indeed the case for most standard choices of $H$, such as probit or logistic sigmoid. To check the Lyapunov condition of Proposition~\ref{proposition:Anh_approxrep} for binary regression models, we note that
\begin{align*}
\E_{\theta}\, \vbb{Z_i}^4\left\{ \frac{\{Y_i - H(Z_i^{\top}\theta)\}H^{\prime}(Z_i^{\top}\theta)}{H(Z_i^{\top}\theta)\{1 - H(Z_i^{\top}\theta)\}} \right\}^4 
& = \E_\theta\left[\vbb{Z_i}^4\,\left\{\frac{H'(Z_i^\top\theta)}{H(Z_i^\top\theta)\{1 - H(Z_i^\top\theta)\}}\right\}^4\E_\theta\left[\{Y_i - H(Z_i^\top\theta)\}^4\given Z_i\right]\right] \\
& = \E_\theta\vbb{Z_i}^4 H'(Z_i^\top\theta)^4\left\{\frac1{H(Z_i^\top\theta)^3} + \frac1{\{1 - H(Z_i^\top\theta)\}^3}\right\} \\
& = \E_\theta\vbb{Z_i}^4 H'(Z_i^\top\theta)^4\frac{1 - 3H(Z_i^\top\theta) + 3H(Z_i^\top\theta)^2}{\left[H(Z_i^\top\theta)\{1 - H(Z_i^\top\theta)\}\right]^3} \\
& \leq 7\,\E_\theta\vbb{Z_i}^4\frac{H'(Z_i^\top\theta)^4}{\left[H(Z_i^\top\theta)\{1 - H(Z_i^\top\theta)\}\right]^3}.
    \notag
\end{align*}
For both the probit and the logistic model, the function $a \mapsto a^4 H^{\prime}(a)^4/\{H(a)(1 - H(a)\}^3$ is bounded, which from the above inequality entails that the Lyapunov condition is satisfied. Similarly, that $a \mapsto a^2 H^{\prime}(a)^2/\{H(a)(1 - H(a)\}$ is bounded, implies that $\max_{i\leq n}\E_{\theta}\,\{ u_{\theta}(Y_i \given X_i)u_{\theta}(Y_i \given X_i)^{\tr}\given \Falg_{i-1}\}$ is bounded in probability. This shows that two of three conditions of Lemma~\ref{proposition:Anh_approxrep} are satisfied for the probit and logistic models. To verify that these two models coupled with certain adaptive designs satisfy all conditions of Theorem~\ref{thm:master_theorem} it therefore remains to show that they are S-DQM and that their predictable quadratic variations converge in probability. For the former task, we employ Lemmas~\ref{lemma:lemma7.6ext2}--\ref{lemma:conv-as}. The latter task must be tackled on a case-by-case basis, and is typically achieved using stationarity or convergence properties of the covariate sequence $X_1, X_2, \dots$.

For the central limit theorem of Theorem~\ref{thm:master_theorem} to be of any use, we need to be able to estimate the limiting variance. For the binary regression models of primary interest in this paper the inverse of the plug-in estimators \begin{equation*}
    \widehat{J}_n = \frac1n\sum_{i=1}^nZ_i Z_i^\top\frac{H'( Z_i^{\top}\widehat{\theta}_n)^2}{H(Z_i^{\top}\widehat{\theta}_n)\{1 - H(Z_i^{\top}\widehat{\theta}_n)\}}
\end{equation*}
may be used to estimate the asymptotic variance. That $\widehat{J}_n$ converges to $J$ in probability for logistic and probit regression follows from consistency of $\widehat{\theta}_n$ for $\theta_0$, upon noting that $a\mapsto H'(a)^2/\{H(a)[1 - H(a)]\}$ is bounded and continuous in these cases.  

\section{Illustration: The Bruceton and Robbins--Monro designs}\label{sec:illustrations}
In this section, we apply Theorem~\ref{thm:master_theorem} to explicitly verify asymptotic normality of the maximum likelihood estimator for two classic adaptive designs in binary regression, namely the Bruceton and Robbins--Monro designs. Both of these are widely used in industrial applications for choosing the inputs $X_1, X_2, \dots$ in a sensitivity testing experiment of the form~\eqref{eq:binarymodel}.
\subsection{Bruceton}\label{sec:bruceton}
Developed by \citet{dixon1948method}, the Bruceton design (also known as the `up-and-down' design) is used by both NATO and the U.S.~Department of Defense for sensitivity testing \citepalias{stanag4487, stanag4488, stanag4489, mil-std-331D}. Due to its simplicity, it is widely employed in both civil and military sector \citep{baker2021overview, collet2021review, esposito2024review}, and also in fields like pharmacology and toxicology \citep{flournoy2002up}. After an initial binary trial at a pre-specified height $X_1 = x_1$, with an observation of $Y_1$, the Bruceton design sets
\begin{equation}\label{eq:bruceton}
    X_{i} = \begin{cases}
                X_{i-1} - d & \mathrm{if}\,\, Y_{i-1} = 1, \\
                X_{i-1} + d & \mathrm{if}\,\, Y_{i-1} = 0,
              \end{cases}
\end{equation}
for $i = 2,3,\ldots, n$, where $d>0$ is a user-defined parameter called the step size. That is, we go down one step if we observe a successful trial, and up one step otherwise. The binary outcomes are modelled as in~\eqref{eq:binarymodel}. Note that there are no noise variables $U_i$ present in the Bruceton design. 

In view of the discussion following Theorem \ref{thm:master_theorem}, it remains to show that S-DQM holds and that $J_n$ converges in order to assert asymptotic normality of the maximum likelihood estimator. By marginalising out the $Y_i$ in \eqref{eq:bruceton}, we see that $(X_i)$ is a homogeneous Markov chain on the discrete space $x_1 + d\integers$, with $\prob(X_1 = x_1) = 1$ and (infinite) transition matrix $P$ defined by 
$$P_{x,x'} = \prob(X_{i} = x' \given X_{i-1} = x) = \begin{cases}H(\theta^\top z) & \mathrm{if}\, x' = x - d, \\ 1 - H(\theta^\top z) & \mathrm{if}\, x' = x + d,\end{cases}$$ 
where we recall that $z = (1, x)^\top$. \citet{derman1957non} first proved that the covariate process $(X_i)$ forms an ergodic Markov chain, and its stationary distribution has been extensively studied \citep{tsutakawa1967random, durham1993convergence, durham1994random, flournoy2002up}. There do exist results asserting asymptotic normality of the maximum likelihood estimator when the Bruceton design is employed \citep{tsutakawa1967asymptotic, fahrmeir1987asymptotic, rosenberger1997asymptotic}, but these all truncate the state space of the covariate sequence to a finite set. Although we believe the techniques developed in these references are powerful enough to handle the original unbounded state space formulation, we nevertheless provide a direct proof here for the sake of completeness.

\begin{lemma}\label{lemma:bruceton-sdqm}
    When equipped with probit or logistic regression, the Bruceton design is S-DQM. 
\end{lemma}
\begin{proof}
    To show S-DQM, we use Lemma~\ref{lemma:lemma7.6ext2}. Let $\nu$ be the counting measure on the integers and $g$ being the constant function 1. Then (i) and (iv) are trivially satisfied. Next, (ii) holds by the Weierstrass M-test and the fact that $J_\theta(x)$ decays exponentially quickly as $x\to\pm\infty$, and (iii) is satisfied by discreteness of the $X_i$. 
\end{proof}
Having shown S-DQM, we need to address the convergence of $J_n$. We do this using standard Markov chain theory. 

\begin{lemma}\label{lemma::mc}
In the Bruceton design, the chain $(X_i)$ is irreducible and positive recurrent.
\end{lemma}

\begin{proof}
Without loss of generality, we can assume that $X_0 = 0$ and $d=1$, so that the state space is $\integers$. Clearly $(X_i)$ is irreducible, as we can reach any state from any other in a finite number of steps with positive probability. To prove positive recurrence, we apply Foster's theorem \citep{foster1953stochastic}. To do so, we need to find a function $V:\integers\to \reals$, a finite set $F\subseteq \integers$ and $\varepsilon > 0$ such that
\begin{enumerate}
    \item $\displaystyle\sum_{x'\in \integers}P_{x,x'}V(x') < \infty$ for all $x\in F$,
    \item $\displaystyle\sum_{x'\in \integers}P_{x,x'}V(x') \leq V(x) - \varepsilon$ for all $x\in \integers\setminus F$.
\end{enumerate}
Let $V(x) = |x|$. Then for both probit and logistic regression we have that $\lim_{x\to\infty} P_{x,x+1} V(x) = 0$. Indeed, for the former, L'H\^{o}pital's rule yields
    $$\lim_{x\to\infty}P_{x,x+1}V(x+1) = \lim_{x\to\infty}[1-\Phi(\alpha + \beta x)](x+1) \\ = \frac{\beta}{\sqrt{2\pi}}\lim_{x\to\infty}\exp[-(\alpha + \beta x)^2/2](x+1)^2 = 0.$$
Similarly, in logistic regression we have
$$\lim_{x\to\infty} P_{x, x+1} V(x+1) = \lim_{x\to\infty}\left\{1- \sigma(\alpha + \beta x)\right\}(x + 1) = \lim_{x\to\infty}\frac{e^x(x+1)^2}{(1 + e^x)^2} = 0.$$
Hence, in both cases, there exists $\varepsilon > 0$ and a positive integer $n_1>0$ such that $P_{x, x+1}V(x+1) \leq 1-\varepsilon$ for all $x \geq n_1$. Similarly, there exists a positive integer $n_2> 0$ such that $P_{x, x-1}V(x-1) \leq 1 - \varepsilon$ for all $x \leq -n_2$. Let $n_0 = \max\{x_1, x_2\}$. Now, let $F$ be the finite set $\{-n_0, -n_0+1, \dots, n_0-1, n_0\}$. Then clearly $V$ satisfies the first condition of Foster's theorem. For the second condition, let $x> n_0$. Then
$$\sum_{x'\in \integers}P_{x,x'}V(x') = P_{x,x-1}(x-1) + P_{x,x+1}(x+1) \leq 1\times (x-1) + (1 - \varepsilon)= V(x) - \varepsilon.
$$
Similarly, for $x<-n_0$, we have $\sum_{y\in \integers}P_{xy}V(y) \leq V(x) - \varepsilon$. This verifies the second condition.
\end{proof}

Having proved that $(X_i)$ is irreducible and positive recurrent, the ergodic theorem for Markov chains applies (see for example \citet[Theorem~1.10.2, p.~53]{norris1997markov}). Hence, for any bounded function $g:\integers\to \reals$, we know that $n^{-1}\sum_{i=1}^n g(X_i)$ converges. In particular, since the stationary distribution is non-degenerate, $J_n$ converges in probability to an invertible matrix $J$.

The ergodicity of the Bruceton design can be viewed as a consequence of its rubber banding effect (as first described by \citet{durham1994random}), namely that the probability of moving one step down increases as $X_n$ gets large, and, conversely, that the probability of moving one step up increases as $X_n$ gets small. To see how this rubber banding ensures ergodicity, consider a reverse Bruceton design, in which we change the signs in  \eqref{eq:bruceton}. That is, we set $X_{i+1} = X_i + d$ if $Y_i=1$ and $X_{i+1} = X_i - d$ if $Y_i=0$. This clearly does not have any rubber banding property, and, by comparing to a symmetric random walk on the integers, we clearly see that $|X_n| \to\infty$ almost surely. Fittingly, $|X_n|\to\infty$ implies that one of the sets $\{i\in \nat : Y_i = 0\}$ and $\{i\in \nat : Y_i = 1\}$ has finite cardinality, and so not even existence of the maximum likelihood estimator is guaranteed.

A question which perhaps has not been fully appreciated in the literature is what happens if we introduce an adaptive step size to the Bruceton design, say, one which decreases with $n$. For example, we could substitute $d = d_i = c/(2i)$ in \eqref{eq:bruceton}, hoping that in addition to ergodicity, we would also have that $(X_i)$ converges almost surely to the median of the distribution function $H$. In fact, making the substitution, we obtain $X_{i} = X_{i-1} - c(Y_{i-1} - 1/2)/(i-1),$ which we recognize as a special case of the Robbins--Monro procedure, to which we shall now direct our attention.

\subsection{Robbins--Monro}
The study of the Robbins--Monro design \citep{robbins1951stochastic} comprises a vast literature spanning more than a half-century (see \citet{lai2003stochastic} for an excellent review containing many fascinating historical remarks). This procedure aims to find the unique root $\gamma$ of an unknown function $M$ under the regression model
\begin{equation}\label{eq:robbins_monro}
    Y_i = M(X_i) + \varepsilon_i,
\end{equation}
where the $\varepsilon_i$ are independent noise variables. \citet{robbins1951stochastic} proposed an adaptive design of the form $h(\{(X_j, Y_j)\}_{1\leq j\leq i-1}) = X_{i-1} - a_{i-1}Y_{i-1}$ where $a_i > 0$ for all $i$, and proved that $X_i\to\gamma$ in probability provided $\sum_{i=1}^\infty a_i^2 < \infty$ and $\sum_{i=1}^\infty a_i = \infty$. In their paper they note that the procedure \eqref{eq:robbins_monro} includes the binary regression setting outlined above. Indeed, suppose we wish to estimate the $q$th quantile $\gamma_q$ of an unknown cumulative distribution function (c.d.f.) $F$  from binary responses. This corresponds to the binary regression model~\eqref{eq:binarymodel}, with $F(x) = H(\theta^\top z)$. Replacing $Y_i$ by $Y_i - q$, we have $M(X_i) = F(X_i) - q$, with $\varepsilon_i + F(X_i)\,|\, X_i\sim\mathrm{Bernoulli}(F(X_i))$. The large-sample properties of the covariate process $(X_i)$ were extensively studied in the decade following their publication, summarised in the review by \citet{schmetterer1961stochastic}. The consistency result for $(X_i)$ was first improved upon by \citet{blum1954approximation}, who proved that $X_i\to\gamma$ almost surely. Using this consistency result, \citet{chung1954stochastic, hodges1956two, sacks1958asymptotic} established asymptotic normality of $X_i$. As noted by \citet{cochran1965robbins}, this yields asymptotic normality of the estimate of the $q$th quantile $\gamma_q$ in the case of binary regression. Further large-sample results were later proved for the Robbins--Monro procedure and variants of it by \citet{lai1979adaptive, lai1981consistency, lai1982adaptive, wei1987multivariate}.

As $X_i\to\gamma$ almost surely in the Robbins--Monro design, the most natural estimator for the median of the underlying distribution function $F$ given $n$ observations is simply $X_{n+1}$, namely the next hypothetical input. Hence, nearly all asymptotic theory for the design concerns the large-sample behaviour of $X_{n+1}$, which is entirely model-independent. If, however, we want to construct a confidence interval for $\gamma$, then a model needs to be imposed. Furthermore, in the context of binary regression, we might want to estimate other quantiles than simply $\gamma = \gamma_{1/2}$, which will also require us to impose a model on the data. Hence, the model-specific large-sample behaviour of the maximum likelihood estimator may also be of interest when the Robbins--Monro design is employed. As with the Bruceton design, asymptotic normality will be established by verifying S-DQM and convergence of $J_n$ in probability. For the former, we may appeal directly to Lemma~\ref{lemma:conv-as}. For the latter, we only require that $X_i\to\gamma$ in probability. Indeed, if this is assumed to hold, then each term of $J_n$ also converges in probability by the continuous mapping theorem. Using the fact that the Fisher information is uniformly integrable (indeed bounded) for probit and logistic regression, we see that $J_n$ also converges in probability by \citet[Remark 3.1]{linero2013toeplitz}.

It is worth exploring why it is not possible to simply apply Lemma~\ref{lemma:lemma7.6ext2} like we did for the Bruceton design in order to verify S-DQM for the Robbins--Monro design. Let us attempt to do so and see what goes wrong. Like before, let $g$ be the constant function 1, and now, let $\nu$ be the counting measure on the (countable) union of the finite domains of $X_i$ for all $i$. Then all conditions are indeed satisfied except for condition (ii). As the sequence $X_1, X_2$ converges, we may no longer apply the Weierstrass M-test to assert continuity of the function $\theta\mapsto \int J_\theta(x)d\nu(x)$. In fact, this function will be infinite everywhere. Somehow, the fact that the domain of the Bruceton design are spaced apart ensures the argument to go through, but we run into problems when the covariate sequence converges or is defined on a compact domain. This demonstrates explicitly why we also need results like Lemmas~\ref{lemma:compact} and \ref{lemma:conv-as}.

\section{A Markovian Langlie design}\label{sec:langlie}
Having applied Theorem~\ref{thm:master_theorem} to the Bruceton design in the previous section, we now move on to a more complicated experimental design, namely a Markovian version of the Langlie design (see \citet{langlie1962reliability} for the original non-Markovian version). Our results in this section may serve as a first step on the way to establishing convergence of $J_n$ for the original Langlie design, which remains an open problem. Our Markovian Langlie design depends on two parameters, $a<b$, serving as lower and upper bounds for the measurements, respectively. These parameters are pre-specified by the researcher. It is assumed that the researcher chooses $a$ and $b$ so that they cover the true, unknown, median. That is, $a < -\alpha/\beta < b$. The first test is conducted at $X_1 = (a+b)/2$, and given the results of the first $i-1$ tests, the input $X_{i}$ is decided by the rule
\begin{equation}\label{eq:markov_langlie}
    X_{i} = \begin{cases} (a + X_{i-1}) / 2 + \varepsilon U_{i-1} & \mathrm{if}\, Y_{i-1} = 1 \\ (X_{i-1} + b)/2 -\varepsilon U_{i-1} & \mathrm{if}\, Y_{i-1} = 0,\end{cases}
\end{equation}
where the $U_i$ are i.i.d.~$\mathrm{Uniform}[0,1]$ random variables and $\varepsilon>0$ is a constant which can be made arbitrarily small. In the remainder of this section, we will set $a=0$ and $b=1$ without loss of generality. By compactness of the domain of the $X_i$, we see immediately that S-DQM is satisfied by Lemma~\ref{lemma:compact}. The remainder of this section is thus devoted to establishing ergodicity of the chain $(X_i)$ in order to verify the convergence of $J_n$, which in turn will imply asymptotic normality of the maximum likelihood estimator.

The presence of the terms $\varepsilon U_{i-1}$ in \eqref{eq:markov_langlie} is purely technical, but necessary to ensure irreducibility of the chain $(X_i)$. Indeed, if $\varepsilon = 0$, then for all $n$, we would have $X_n \in \mathcal{D}_n = \{k/2^n : 1\leq k \leq 2^n-1,\, k\,\,\mathrm{odd}\}$. As $\mathcal{D}_n\cap\mathcal{D}_m = \emptyset$ for all $n\neq m$, it would be impossible for the chain to visit any previously visited state, making it reducible. 

The kernel for the Markovian Langlie is given by
\begin{equation}\label{eq:ctslanglie}
k(x, \mathrm{d} x') = H_x\frac1\varepsilon\mathbbm{1}_{\left[\frac{x}2, \frac{x}2+\varepsilon\right]}(\mathrm{d} x') + \left[1 - H_x\right]\frac1\varepsilon\mathbbm{1}_{\left[\frac{x+1}2-\varepsilon, \frac{x+1}2\right]}(\mathrm{d} x'),
\end{equation}
where $H_x = H(\alpha + \beta x) = H(\theta^\top z)$ and we assume that $0 < \varepsilon < 1/2$, so that the two intervals $\left[\frac{x}2, \frac{x}2+\varepsilon\right]$ and $\left[\frac{x+1}2-\varepsilon, \frac{x+1}2\right]$ do not overlap.

\begin{lemma}\label{lemma:irreducible}
    The chain defined in \eqref{eq:markov_langlie} is irreducible with respect to the Lebesgue measure on $[0, 1]$ equipped with the Borel $\sigma$-algebra.
\end{lemma}

\begin{proof}
    See the Appendix.
\end{proof}

The next step in proving ergodicity is to prove that the chain emits a stationary measure. Since the state space $[0, 1]$ is compact, the easiest way to show this is to exploit the continuity properties of the kernel $k$. More precisely, $k$ is \emph{weakly Feller}, which means that for any bounded function $g:[0, 1]\to \reals$, the function
$$x\mapsto \int_0^1 g(y)k(x, \mathrm{d} y)$$
is also bounded and continuous. Indeed, this follows from the fact that $H_x$ is continuous in $x$ and the functions 
$$
x \mapsto \int_{x/2}^{x/2 + \varepsilon} g(y)\,\mathrm{d} y\quad\text{and}\quad x\mapsto \int_{\frac{x+1}2 - \varepsilon}^{\frac{x+1}2}g(y)\,\mathrm{d} y, 
$$
are continuous. By \citet[Theorem 3.3]{lasserre1997invariant}, the kernel $k$ has an invariant probability measure.

\begin{remark}
    In fact, the chain is also \emph{strongly Feller}, since the transition kernel also maps all bounded measurable functions to bounded continuous functions by virtue of compactness of $[0, 1]$. However, for our purposes, the weak Feller property suffices.
\end{remark}

Since we need ergodicity for the specific starting point $X_1 = 1/2$ (and not just for almost all starting points), we also need to show that $(X_i)$ is Harris recurrent. We will do this by providing a petite recurrent set. This requires a lemma.

\begin{lemma}\label{lemma:drift_function}
    There exists $\varepsilon > 0$ and $m$ (depending on $\varepsilon$) such that for the Markovian Langlie design, the function $V(x) = mx + 1$ satisfies the Lyapunov drift condition. That is,
        $$\int V(y) k(x, \mathrm{d} y) - V(x) + 1 < 0$$
    for all $x$ in some open neighbourhood $1$. Here, open neighbourhood means a set of the form $(\eta, 1]$ for some $\eta\in [0, 1)$.
\end{lemma}

\begin{proof}
    See the Appendix.
\end{proof}

Having found a Lyapunov function for the chain $(X_i)$, the set $[0, \eta]$ is recurrent by \citet[Proposition~7.12, p.~147]{benaim2022markov}. Furthermore, by irreducibility, the set $[0, \eta]$ is also petite by Proposition~6.2.8~(ii) in \citet{meyn1993markov}. This means that we have identified a petite recurrent set, and so the chain is Harris recurrent by \citet[Proposition~7.11, p.~147]{benaim2022markov}. Finally, Theorem~17.3.2 in \citet{meyn1993markov} shows that the chain $(X_i)$ is ergodic. We conjecture that the chain is also ergodic if we set $\varepsilon_i=0$ for all $i$, but further research into perturbed Markov kernels is required to show this, as discussed by \citet{shardlow2000perturbation, alquier2014noisy}.

\section*{Acknowledgements}
We thank Professor Nancy Flournoy and Dr.~Benjamin Lang for useful and enlightening discussions.

\appendix

\setcounter{equation}{0}
\renewcommand{\theequation}{S.\arabic{equation}}
\renewcommand{\thelemma}
{S.\arabic{lemma}}

\section*{Appendix}

\subsection*{Proof of Lemma~\ref{lemma:lemma7.6ext2}}
\begin{proof} We first prove that Conditions~(i) and~(ii) imply that $\int D_{h/\sqrt{n}}(x)g(x) \,\dd\nu(x) = o(1/n)$. This part of the proof is essentially the same as the proof of Lemma~7.6 in \citet{vandervaart1998asymptotic}. Define $\bar{J}_{\theta} = \int J_{\theta}(x) g(x)\,\dd\nu(x)$ and $s_{\theta}(y \given x) = \sqrt{f_{\theta}(y \given x)}$, and let $\dot{s}_{\theta}(y \given x)$ be the derivative of $s_{\theta}(y \given x)$ with respect to $\theta$. Since $s_{\theta}$ is differential at $\theta_0$ for all $(x,y)$, $\dot{s}_{\theta}(y \given x) = \half u_{\theta_0}(y \given x) \sqrt{f_{\theta}(y \given x)}$, where $u_{\theta_0} = \dot{f}_{\theta_0}/f_{\theta_0}$, and $\dot{f}_{\theta}$ is the derivative of $f_{\theta}(y \given x)$ with respect to $\theta$. By Condition~(i) 
\begin{equation}
\frac{s_{\theta_0 + th } - s_{\theta_0}}{t} = \int_0^1 \dot{s}_{\theta_0 + u th} \,\dd u. 
    \notag
\end{equation}
Viewing this as an expectation of $\dot{s}_{\theta_0 + u th}(y \given x)$ with respect to the uniform measure $\dd u$ on the unit interval, Jensen{'}s inequality yields
\begin{equation}
\left\{\frac{s_{\theta_0 + th } - s_{\theta_0}}{t}\right\}^2 \leq \int_0^1 \left\{ h^{\tr}\dot{s}_{\theta_0 + u th}\right\}^2 \,\dd u = \quart\int_0^1 \left\{ h^{\tr}u_{\theta_0 + uth } \right\}^2 f_{\theta_0+uth} \,\dd u. 
    \notag
\end{equation}
Integrating with respect to $\mu$, then applying Tonelli{'}s theorem, applicable because $\{h^{\tr}u_{\theta} \sqrt{f_{\theta}}\}^2 \geq 0$ for all $(x,y)$ and $\theta$ close to $\theta_0$, we get
\begin{equation}
\int \left\{ \frac{s_{\theta_0 + th } - s_{\theta_0}}{t}\right\}^2\,\dd\mu(y) 
\leq \quart \int \int_0^1 
\left\{h^{\tr}u_{\theta_0 + uth } \right\}^2 f_{\theta_0+uth} \,\dd u\,\dd\mu(y) = \quart
\int_0^1 h^{\tr}J_{\theta_0 + uth }(x) h \,\dd u,
    \notag
\end{equation}
and yet another application of Tonelli{'}s theorem yields
\begin{equation}
\int \int \left\{ \frac{s_{\theta_0 + th } - s_{\theta_0}}{t}\right\}^2\,\dd\mu(y) g(x) \,\dd\nu(x) 
\leq \quart
\int_0^1 h^{\tr}\bar{J}_{\theta_0 + uth } h \,\dd u.
    \notag
\end{equation}
By Condition~(ii), $\theta \mapsto \bar{J}_{\theta}$ is continuous, which by the developments so far entails that
\begin{equation}
\lim_{t \to 0}\int \int \left\{ \frac{s_{\theta_0 + th } - s_{\theta_0}}{t}\right\}^2\,\dd\mu(y) g(x) \,\dd\nu(x) 
\leq \quart
\int_0^1 h^{\tr}\bar{J}_{\theta_0} h \,\dd u.
    \notag
\end{equation}
By the integrability of $s_{\theta}$, $t^{-1}(s_{\theta_0 + th} - s_{\theta_0}) \to \dot{s}_{\theta_0}$ for each $(x,y)$. Combining this with the bounds above, Proposition~2.29 in~\citet{vandervaart1998asymptotic} gives
\begin{equation}
\int \int \left\{\frac{s_{\theta_0 + th}(y \given x) - s_{\theta_0}(y \given x)}{t} - h^{\tr}\dot{s}_{\theta_0}(y \given x)  \right\}^2\,\dd\mu(y) g(x)\,\dd\nu(x) \to 0, \quad \text{as $t \to 0$},
    \notag
\end{equation}
and we conclude that $\int D_{\theta_0,h/\sqrt{n}}(x) g(x)\,\dd\nu(x) = o(1/n)$, as desired. To get from here to S-DQM, we use Conditions~(iii) and (iv). These say that $\{P_{\theta_0}^{X_1},P_{\theta_0}^{X_2},\ldots\}$ is dominated by the $\sigma$-finite measure $\nu$, and that the densities $\dd P_{\theta_0}^{X_i}/\dd \nu$ are bounded by $g$. Since $D_{\theta_0,h/\sqrt{n}}(x) \geq 0$ for all $n$ and $x$,  
\begin{align*}
n\max_{i \leq n}\E_{\theta_0}\,D_{\theta_0,h/\sqrt{n}}(X_i) & = 
n \max_{i \leq n} \int D_{\theta_0,h/\sqrt{n}}(x)\,\dd P_{\theta_0}^{X_i}(x)\\
& = n \max_{i \leq n} \int D_{\theta_0,h/\sqrt{n}}(x)\,\frac{\dd P_{\theta_0}^{X_i}}{\dd \nu}(x)\,\dd \nu(x) 
\leq  n  \int D_{h/\sqrt{n}}(x)g(x)\,\dd \nu(x),
\end{align*}
where the right hand side is $o(1)$ by the above. Since $D_{\theta,h/\sqrt{n}}$ is non-negative, it now follows from Markov{'}s inequality that $\sum_{i=1}^n D_{\theta_0,h/\sqrt{n}}(X_i) = o_p(1)$.
\end{proof}

\subsection*{Proof of Lemma~\ref{lemma:compact}}
\begin{proof}
    Take $\theta$ in the interior of $\Theta_0$ and choose $h$ sufficiently small such that $\theta + hz\in \Theta_0$ for any $0\leq z\leq 1$. Write $s_\theta(y\given x) = \sqrt{f_\theta(y\given x)}$ and $g_{n, \theta}(x) = nD_{\theta, h/\sqrt{n}}(x)$. For any $z\in [0, 1]$, there exists by the mean value theorem $\tilde{z}$ between $0$ and $z$ such that $$\left.\frac{s_{\theta + hz}(y\given x) - s_\theta(y\given x)}{z} = \frac{\partial}{\partial z}s_{\theta + hz}(y\given x)\right|_{z={\tilde{z}}}.$$
    Hence 
    \begin{align*}
        |g_{n, \theta}(x)| & =  \sum_{y\in\yspace}\left\{\frac{s_{\theta + h/\sqrt{n}}(y\given x) - s_\theta(y\given x)}{1/\sqrt{n}} - \frac12 h^\top u_\theta(y\given x) s_\theta(y\given x)\right\}^2 \\
        & = \sum_{y\in\yspace}\left\{\left.\frac{s_{\theta + h/\sqrt{n}}(y\given x) - s_\theta(y\given x)}{1/\sqrt{n}} - \frac{\partial}{\partial z}s_{\theta + hz}(y\given x)\right|_{z=0}\right\}^2 \\
        & = \sum_{y\in\yspace}\left\{\left.\frac{\partial}{\partial z}s_{\theta+hz}(y\given x)\right|_{z=\tilde z} - \left.\frac{\partial}{\partial z}s_{\theta + hz}(y\given x)\right|_{z=0}\right\}^2.
        \end{align*}
    By compactness of $\xspace\times\Theta_0$ and continuity of $\frac{\partial}{\partial\theta}s_\theta(y\given x)$, we therefore have
    $$|g_{n,\theta}(x)| \leq 2\sum_{y\in\yspace}\sup_{\theta\in\Theta_0}\left|\frac{\partial}{\partial \theta} h^\top s_\theta(y\given x)\right|^2 \leq 2\sum_{y\in\yspace}\sup_{x\in\xspace}\sup_{\theta\in\Theta_0}\left|\frac{\partial}{\partial\theta} h^\top s_\theta(y\given x)\right|^2 = M,$$
    so for each $\theta\in\Theta_0$, the sequence $(g_{n, \theta})$ is uniformly bounded on the compact domain $\xspace$. Note that by the Cauchy--Schwarz inequality, we also have the uniform bound $$\sum_{y\in\yspace}\left|\frac{s_{\theta + h/\sqrt{n}}(y\given x) - s_\theta(y\given x)}{1/\sqrt{n}} - \frac12 h^\top u_\theta(y\given x) s_\theta(y\given x)\right| \leq \sqrt{k}M.$$
    Next, we aim to show that for each $\theta$, the sequence $(g_{n, \theta}$) is equicontinuous in $x$. We do this by bounding the functions $\frac{\partial}{\partial x}g_{n, \theta}(x)$ uniformly. Now, by the chain rule,
    \begin{align*}
        \left|\frac{\partial}{\partial x}g_{n,\theta(x)}\right| & = \sum_{y\in\yspace}\frac{\partial}{\partial x}\left\{\left.\frac{\partial}{\partial z}s_{\theta + hz}(y\given x)\right|_{z=\tilde z} - \left.\frac{\partial}{\partial z} s_{\theta + hz}(y\given x)\right|_{z=0}\right\}^2 \\
        & \leq 2\sqrt{k}M\sum_{y\in\yspace}\left|\left\{\left.\frac{\partial}{\partial x}\frac{\partial}{\partial z} s_{\theta + hz}(y\given x)\right|_{z=\tilde z} - \left.\frac{\partial}{\partial x}\frac{\partial}{\partial z}s_{\theta + hz}(y\given x)\right|_{z=0}\right\}\right|,
    \end{align*}
    and so by compactness of $\xspace\times\Theta_0$ and continuity of $\frac{\partial}{\partial x}\frac{\partial}{\partial\theta} s_\theta(y\given x)$, we get
    $$\left|\frac{\partial}{\partial x} g_{n, \theta}(x)\right| \leq 4\sqrt{k} M \sum_{y\in\yspace}\sup_{x\in\xspace}\sup_{\theta\in\Theta_0}\left|\frac{\partial}{\partial x}\frac{\partial}{\partial \theta} h^\top s_\theta(y\given x)\right| = L,$$
    providing the required uniform bound. We conclude that $(g_{n, \theta})$ is equicontinuous in $x$.

    Now, write $a_n = \sum_{i=1}^n\E\,D_{\theta, h/\sqrt{n}}(X_i)$ and let $(a_{n_k})$ be any subsequence of $(a_n)$. By the above, the subsequence $(g_{n_k, \theta})$ is uniformly bounded and equicontinuous on a compact domain, so by the Arzel\`{a}--Ascoli theorem, there exists a subsequence $(g_{n_{k_j}, \theta})$ which converges uniformly on $\xspace$. We know that $g_{n, \theta}$ converges pointwise to 0, so the uniform limit of $(g_{n_{k_j}, \theta})$ must also be 0. But then
    $$|a_{n_{k_j}}| \leq n_{k_j}\sup_{x\in\xspace}|D_{\theta, h/\sqrt{n_{k_j}}}(x)| = \sup_{x\in\xspace}|g_{n_{k_j}}(x)| \to 0,$$
    so we found a subsequence of $(a_{n_k})$ converging to zero. Since $(a_{n_k})$ was arbitrary, this means that $a_n\to 0$ as well.
\end{proof}

\subsection*{Proof of Lemma~\ref{lemma:conv-as}}
We begin with a corollary to Lemma~\ref{lemma:compact}
\begin{corollary*}\label{cor:deterministic}
    Suppose all the conditions of Lemma~\ref{lemma:compact} are satisfied. Then for any deterministic sequence $(x_i)$ in $\xspace$, $\sum_{i=1}^n D_{\theta, h/\sqrt{n}}(x_i) \to 0$.
\end{corollary*}

\begin{proof}
    This follows from letting $\prob(X_i = x_i) = 1$ for all $i$ and then applying Lemma~\ref{lemma:compact}.
\end{proof}
\begin{proof}[Proof of Lemma~\ref{lemma:conv-as}]
    Let $\Omega$ be the sample space on which the sequence $(X_i)$ is defined. Now, take any $\omega\in\Omega$ such that $X_i(\omega)\to\gamma$. Then the (deterministic) sequence $X_i(\omega)$ is bounded, so we can find a compact set $\xspace_\omega$ such that $X_i(\omega) \in \xspace_\omega$ for all $i$. By the above Corollary, we thus have $\sum_{i=1}^n D_{\theta, h/\sqrt{n}}(X_i(\omega))\to 0$. We thus have
    $$\{\omega\in\Omega : X_i(\omega)\to\gamma\} \subseteq \{\omega\in\Omega : \sum_{i=1}^n D_{\theta, h/\sqrt{n}}(X_i(\omega)) \to 0\},$$
    which, combined with almost sure convergence of $(X_i)$, implies that $\prob(\sum_{i=1}^n D_{\theta, h/\sqrt{n}}(X_i) \to 0) = 1$, as required. The final statement follows from convergence in probability being weaker than almost sure convergence.
\end{proof}

\subsection*{Proof of Proposition~\ref{proposition:Anh_approxrep}}
\begin{proof} 
To save space we write $D_{h/\sqrt{n}} = D_{\theta,h/\sqrt{n}}$. Following the proof of Theorem~7.2 in \citet{vandervaart1998asymptotic}, define
\begin{equation}
W_{n,i} = 2\left\{\sqrt{f_{\theta + h/\sqrt{n}}(Y_i \given X_i)/f_{\theta}(Y_i \given X_i)} - 1\right\}.
\notag
\end{equation}
For $|w| < 1$, Taylor{'}s theorem yield $\log(1 + w) = w - \half w^2/(1 + \tilde{w})^2 = w - \half w^2 + w^2 \half( 1 - 1/(1 + \tilde{w})^2) = w - \half w^2 + w^2 R(2w)$, where $\tilde{w} = \tilde{w}(w)$ is some number between $0$ and $w$, and $R(w) = \half(1 - 1/(1 + \tilde{w}(w/2))^2)$. Then $R(w) \to 0 $ as $w \to 0$, and    
\begin{equation}
\begin{split}
A_n(h) & = \sum_{i=1}^n \log \frac{f_{\theta + h/\sqrt{n}}(Y_i \given X_i)}{f_{\theta}(Y_i \given X_i)}
= \sum_{i=1}^n 2 \log ( \half W_{n,i} + 1)\\
& = \sum_{i=1}^n W_{n,i} - \quart \sum_{i=1}^n W_{n,i}^2 + \half\sum_{i=1}^n W_{n,i}^2R( W_{n,i} ). 
\end{split}
\label{eq::AnwithTaylor}
\end{equation} 
Introduce the shorthands $f_{\theta,i} = f_{\theta}(y \given X_i)$, $u_{\theta,i} = u_{\theta}(y \given X_i)$, and $\dd \mu = \dd \mu(y)$, and recall that  
\begin{equation}
\E_{\theta}\, [(h^{\top} u_{\theta,i})^2\given \Falg_{i-1}]
= \int h^{\top} u_{\theta,i}u_{\theta,i}^{\top}h f_{\theta,i}\,\dd\mu 
= n h^{\top}(\langle U_{n,\cdot},U_{n,\cdot}\rangle_{i}
- \langle U_{n,\cdot},U_{n,\cdot}\rangle_{i-1})h.
\notag
\end{equation}
We show that $\sum_{i=1}^n W_{n,i}$ and $\sum_{i=1}^n W_{n,i}^2$ converge in probability to the appropriate limits, and that $\sum_{i=1}^n W_{n,i}^2 R(W_{n,i})$ converges to zero in probability. To this end, introduce the processes
\begin{equation}
\begin{split}
\xi_{n,i} & = W_{n,i}  - n^{-1/2}h^{\top}u_{\theta,i} -\E_{\theta}[W_{n,i}\given \Falg_{i-1}],\\
\zeta_{n,i} & = n^{-1}\{(h^{\top}u_{\theta,i})^2
 - \E_{\theta}\, [(h^{\top}u_{\theta,i})^2 \given \Falg_{i-1}]\},\\
\eta_{n,i} & = W_{n,i}^2 - n^{-1}(h^{\top}u_{\theta,i})^2,
\end{split}
\notag
\end{equation}
so that $W_{n,i}^2 = n^{-1}\E_{\theta}[(h^{\top}u_{\theta,i})^2\given \Falg_{i-1}] + \zeta_{n,i} + \eta_{n,i}$; and we notice that $\xi_{n,i}$ and $\zeta_{n,i}$ are martingale differences. Define the martingales $M_{n,j} = \sum_{i=1}^j \xi_{n,i}$ and $Z_{n,j} = \sum_{i=1}^j \zeta_{n,i}$ for $j=1,\ldots,n$, $n \geq 1$.
With this notation introduced we can write
\begin{align}
A_n(h) & = \sum_{i=1}^n\left\{\xi_{n,i} + n^{-1/2}h^\top u_{\theta, i} + \E_\theta[W_{n, i}\given\Falg_{i-1}]\right\} - \quart\sum_{i=1}^n\left\{\eta_{n,i} + \zeta_{n, i} + \E_\theta[n^{-1}(h^\top u_{\theta, i})^2\given\Falg_{i-1}]\right\} \\
& = h^\top U_{n,n} + \sum_{i=1}^n \E_{\theta}[W_{n,i}\given \Falg_{i-1}]  - \quart h^\top\langle U_{n,\cdot},U_{n,\cdot}\rangle_n h
+ M_{n,n} - \quart Z_{n,n}
- \quart \sum_{i=1}^n\eta_{n,i} + r_{n},
\label{app:Anhprocess}
\end{align} 
where $r_n = \half\sum_{i=1}^n W_{n,i}^2 R(W_{n,i})$. The conditional expectation appearing in~\eqref{app:Anhprocess} is 
\begin{equation}
\begin{split}
\E_{\theta}[W_{n,i}\given \Falg_{i-1}] & = \int 2\left\{\sqrt{f_{\theta + h/\sqrt{n}}}/\sqrt{f_{\theta,i}} - 1\right\}f_{\theta,i}\,\dd \mu
 \\
 & = 2\left\{\int  \sqrt{f_{\theta + h/\sqrt{n},i}\times f_{\theta,i}}
\,\dd\mu- 1\right\}\\
& =-\int \left\{\sqrt{f_{\theta + h/\sqrt{n},i}} - \sqrt{f_{\theta,i}}\right\}^2\,\dd\mu\\
&  =  -\int \left\{\sqrt{f_{\theta + h/\sqrt{n},i}} - \sqrt{f_{\theta,i}} - \half n^{-1/2}h^\top u_{\theta,i}\sqrt{f_{\theta,i}}
+ \half n^{-1/2}h^\top u_{\theta,i}\sqrt{f_{\theta,i}} \right\}^2\,\dd\mu\\
&  = -\frac{1}{4n}\E_{\theta}\,[(h^{\top}u_{\theta,i})^2\given \Falg_{i-1}] - D_{h/\sqrt{n}}(X_i)\\
&  \qquad\qquad 
- n^{-1/2}\int \left\{\sqrt{f_{\theta + h/\sqrt{n},i}} - \sqrt{f_{\theta,i}} - \half n^{-1/2}h^\top u_{\theta,i}\sqrt{f_{\theta,i}}\right\} h^\top  u_{\theta,i} \sqrt{f_{\theta,i}}\,\dd\mu.\\
\end{split}
\notag
\end{equation}
For the last term on the right hand side here, the (conditional) H{\"o}lder{'}s inequality yields
\begin{equation}
n^{-1/2}\int | \sqrt{f_{\theta + h/\sqrt{n},i}} - \sqrt{f_{\theta,i}} - \half n^{-1/2}h^\top u_{\theta,i}\sqrt{f_{\theta,i}}| |h^\top  u_{\theta,i} \sqrt{f_{\theta,i}}|\,\dd\mu \leq n^{-1/2}\sqrt{D_{h/\sqrt{n}}(X_i)  \,\E_{\theta}[(h^{\top}u_{\theta,i})^2\given \Falg_{i-1}]},
\notag
\end{equation}
and when we sum this up and use the Cauchy--Schwarz inequality,
\begin{equation}
n^{-1/2}\sum_{i=1}^n \sqrt{D_{h/\sqrt{n}}(X_i)\,\E_{\theta}[(h^{\top}u_{\theta,i})^2\given \Falg_{i-1}]} 
\leq
\sqrt{\sum_{i=1}^nD_{h/\sqrt{n}}(X_i)\langle U_{n,\cdot},U_{n,\cdot}\rangle_n} =
o_p(1)
\label{eq:cleverJensen}
\end{equation}
since $\sum_{i=1}^nD_{h/\sqrt{n}}(X_i) = o_p(1)$ by Lemma~\ref{lemma:lemma7.6ext2} and using that $\langle U_{n,\cdot},U_{n,\cdot}\rangle_n = O_p(1)$ by assumption. We have then shown that $\sum_{i=1}^n \E_{\theta}\, [W_{n,i}\given \Falg_{i-1}] = -\quart h^\top\langle U_{n,\cdot},U_{n,\cdot}\rangle_n h + o_p(1)$, and we can write
\begin{equation}
A_n(h) = h^\top U_{n,n}   - \half h^\top\langle U_{n,\cdot},U_{n,\cdot}\rangle_nh 
+ M_{n,n} - \quart Z_{n,n}
- \quart \sum_{i=1}^n\eta_{n,i} + r_{n} + o_p(1),
\notag
\end{equation} 
It remains to show that $M_{n,n} - \quart Z_{n,n} + \sum_{i=1}^n \zeta_{n,i} + r_n$ is $o_p(1)$. The conditional variance of $\xi_{n,i}$ is 
\begin{equation}
\begin{split}
\E_{\theta}[\xi_{n,i}^2\given \Falg_{i-1}]  
 & \leq \E_{\theta}\, [(  W_{n,i} - n^{-1/2}  h^\top  u_{\theta,i} )^2 \given \Falg_{i-1}]\\
&   = \int \left\{2 \left(\sqrt{f_{\theta + h/\sqrt{n},i}/f_{\theta,i}} - 1\right)
 - n^{-1/2} h^\top  u_{\theta,i} \right\}^2f_{\theta,i} \,\dd\mu\\
&   =
 4\int \left\{ \sqrt{f_{\theta + h/\sqrt{n},i}} - \sqrt{f_{\theta,i}}
 - \half n^{-1/2}h^\top   u_{\theta,i}\sqrt{f_{\theta,i}} \right\}^2 \,\dd\mu\\
 & = 4 D_{h/\sqrt{n}}(X_i).
\end{split}
\label{eq:appendix_xisqbound}
\end{equation}
Thus, $\E_{\theta}\, \xi_{n,i}^2 \leq 4 \E_{\theta}\,D_{h/\sqrt{n}}(X_i)$ which is finite by the S-DQM assumption (at least for big enough $n$) showing that $\xi_{n,i}^2$ is square integrable, and so $M_{n,j}$ is a square integrable martingale. The process $j \mapsto \langle M_{n,\cdot},M_{n,\cdot}\rangle_j$ is predictable, and $\E_{\theta}\, M_{n,n}^2 
= \E_{\theta}\,\sum_{i=1}^n \E_{\theta}\,[\xi_{n,i}^2\given \Falg_{i-1}]
= \E_{\theta}\, \langle M_{n,\cdot},M_{n,\cdot}\rangle_n$. Lenglart{'}s inequality for discrete time martingales (see, e.g., \citet[Theorem~VII.3.4, p.~496]{shiryaev1996probability}) then gives that for any $\eps,\delta > 0$
\begin{equation}
\prob_{\theta}(\max_{j \leq n}|M_{n,j}| \geq \eps)
\leq \frac{\delta}{\eps^2} 
+ \prob_{\theta}( \langle M_{n,\cdot},M_{n,\cdot}\rangle_n \geq \delta), 
\leq 
\frac{\delta}{\eps^2} 
+ \prob_{\theta}( 4\sum_{i=1}^n D_{h/\sqrt{n}}(X_i) \geq \delta),
    \notag
\end{equation}
which tends to $\delta/\eps^2$ as $n \to \infty$ since $\sum_{i=1}^n D_{h/\sqrt{n}}(X_i) = o_p(1)$. But since $\delta>0$ and $\eps>0$ were arbitrary, we conclude that $M_{n,n} = o_p(1)$. 

In order to show that $Z_{n,n} = \sum_{i=1}^n \zeta_{n,i}$ is $o_p(1)$ we use the same type of argument. By assumption, there is a $K$ such that $\E_{\theta}\,\vbb{u_{\theta,i}}^4 \leq K$ for all $i \geq 1$. Therefore 
\begin{equation}
\E_{\theta}\, \zeta_{n,i}^2 
\leq \E_{\theta}\, (h^{\top}u_{\theta,i})^4 
\leq \vbb{h}^4 \, \E_{\theta} \,\vbb{u_{\theta,i}}^4 \leq \vbb{h}^4 K,
\notag
\end{equation}
for all $i \geq 1$. The martingale $Z_{n,j}$ is therefore square integrable with predictable quadratic variation process
\begin{equation}
j \mapsto \langle Z_{n,\cdot},Z_{n,\cdot}\rangle_{j} = 
\frac{1}{n^2}\sum_{i=1}^j \var_{\theta}( (h^{\top}u_{\theta,i})^2 \given \Falg_{i-1}).
\notag
\end{equation}
By Markov{'}s inequality $\prob_{\theta}\{  
n^{-1}\sum_{i=1}^n \var_{\theta}( (h^{\top}u_{\theta,i})^2 \given \Falg_{i-1}) \geq M\}
\leq (nM)^{-1}\sum_{i=1}^n \E_{\theta}\,[(h^{\top}u_{\theta,i})^4]
\leq M^{-1}\vbb{h}^4 K$
which shows that $n \langle Z_{n,\cdot},Z_{n,\cdot}\rangle_n = O_p(1)$. Consequently $\langle Z_{n,\cdot},Z_{n,\cdot}\rangle_n = O_p(1/n) = o_p(1)$, and Lenglart{'}s inequality gives that $Z_{n,n}$ is $o_p(1)$, as desired. 

We now show that $\sum_{i=1}^n \eta_{n,i}$ is $o_p(1)$. Using H{\"o}lder{'}s inequality, followed by the bound in \eqref{eq:appendix_xisqbound} and Minkowski{'}s inequality, we have
\begin{align*}
\E_{\theta}\, [|\eta_{n,i}| \given \Falg_{i-1}] 
& = \E_{\theta}\,[ |W_{n,i}^2 - n^{-1}(h^{\tr}u_{\theta,i})^2|\given \Falg_{i-1}]
\\
& = \E_{\theta}\,[ |(W_{n,i} - n^{-1/2}h^{\tr}u_{\theta,i})(W_{n,i} + n^{-1/2}h^{\tr}u_{\theta,i})| \given \Falg_{i-1}]\\
& \leq 
\sqrt{\E_{\theta}\, [(W_{n,i} - n^{-1/2}h^{\tr}u_{\theta,i})^2 \given \Falg_{i-1}]\,\E_{\theta}\,[ (W_{n,i} + n^{-1/2}h^{\tr}u_{\theta,i})^2\given \Falg_{i-1}] }\\
& \leq 2 \sqrt{D_{h/\sqrt{n}}(X_i)\,\E_{\theta}\,[ (W_{n,i} - n^{-1/2}h^{\tr}u_{\theta,i} + 2n^{-1/2}h^{\tr}u_{\theta,i})^2\given \Falg_{i-1}]}\\
& \leq 2 \sqrt{D_{h/\sqrt{n}}(X_i)} 
\left(\sqrt{\E_{\theta}\, [(W_{n,i} - n^{-1/2}h^{\tr}u_{\theta,i})^2 \given \Falg_{i-1}]} + 2n^{-1/2} \sqrt{\E_{\theta}\,[(h^{\top}u_{\theta,i})^2 \given \Falg_{i-1}]
}\right)\\
& \leq 4 D_{h/\sqrt{n}}(X_i) + 4 n^{-1/2}\sqrt{D_{h/\sqrt{n}}(X_i)\,\E_{\theta}\,[(h^{\top}u_{\theta,i})^2 \given \Falg_{i-1}]}.
\end{align*}
Combining Lemma~\ref{lemma:lemma7.6ext2} and Eq.~\eqref{eq:cleverJensen} we see that $\sum_{i=1}^n \E_{\theta}\, [|\eta_{n,i}| \given \Falg_{i-1}] = o_p(1)$. This process is increasing and predictable, and dominates $\sum_{i=1}^j |\eta_{n,i}|$ in the sense that $\E_{\theta}[|\sum_{i=1}^j \eta_{n,i}|] = \E_{\theta}\,\sum_{i=1}^n \E_{\theta}\, [|\eta_{n,i}| \given \Falg_{i-1}]$ (see the definition in \citet[p.~496]{shiryaev1996probability}). By Lenglart{'}s inequality we then have that for any $\eps,\delta > 0$
\begin{equation}
\prob_{\theta}(|\sum_{i=1}^n \eta_{n,i}| \geq 
\eps) 
\leq \prob_{\theta}(\sum_{i=1}^n |\eta_{n,i}| \geq \eps) \leq \frac{\delta}{\eps} + \prob_{\theta}( \sum_{i=1}^n \E_{\theta}\, [|\eta_{n,i}| \given \Falg_{i-1}] \geq \delta    ) ,
\notag
\end{equation}
from which we conclude that $\sum_{i=1}^n \eta_{n,i} = o_p(1)$.
Lastly, we must show that $r_n = o_p(1)$. Suppose that $\max_{i \geq 1} |W_{i\leq n}| = o_p(1)$ as $n \to \infty$. Under this assumption, we have that $\max_{i \leq n}|R(W_{n,i})| = o_p(1)$, and therefore $|r_n|\leq \max_{i \leq n}|R(W_{n,i})|\sum_{i=1}^n W_{n,i}^2 = o_p(1) \sum_{i=1}^n W_{n,i}^2$. Here, the process $j \mapsto \sum_{i=1}^j W_{n,i}^2$ is dominated (in the sense described above) by the increasing predictable process $j \mapsto \sum_{i=1}^j \E_{\theta}\,[W_{n,i}^2 \given \Falg_{i-1}]$. We have that
\begin{equation}
\begin{split}
\E_{\theta}[W_{n,i}^2\given \Falg_{i-1}] &= 4 \int\left\{ \sqrt{f_{\theta + h/\sqrt{n},i}/f_{\theta,i}} - 1  \right\}^2 f_{\theta,i}\,\dd\mu\\
& = 4 \int\left\{ \sqrt{f_{\theta + h/\sqrt{n},i}
}- \sqrt{f_{\theta,i}}  \right\}^2 \,\dd\mu
= -4\,\E_{\theta}[W_{n,i}\given \Falg_{i-1}],
\end{split}
\notag
\end{equation}
and since $\sum_{i=1}^n \E_{\theta}[W_{n,i}\given \Falg_{i-1}] = -\quart h^{\top}\langle U_{n,\cdot},U_{n,\cdot}\rangle_nh + o_p(1)$ where, by assumption, $h^{\top}\langle U_{n,\cdot},U_{n,\cdot}\rangle_n h$ is bounded by $h^{\top} \max_{i \leq n}\E_{\theta_0}\,\{u_{\theta,i}u_{\theta,i}^{\tr}\given \Falg_{i-1}\}h = O_p(1)$, we have that $\sum_{i=1}^n \E_{\theta}[W_{n,i}\given \Falg_{i-1}] = O_p(1)$. Thus, by the identity above $\sum_{i=1}^n \E_{\theta}[W_{n,i}^2\given \Falg_{i-1}] = O_p(1)$ as $n \to \infty$. By Lenglart{'} inequality we may then, given $\eps > 0$ find $K_1,K_2$ such that 
\begin{equation}
\prob_{\theta}( \sum_{i=1}^n W_{n,i}^2 \geq K_1) \leq \frac{K_2}{K_1}
+ \prob_{\theta}( \sum_{i=1}^n \E_{\theta}\,[W_{n,i}^2 \given \Falg_{i-1}] \geq K_2) < \eps, 
\notag
\end{equation}
for $n$ large enough, which is to say that $\sum_{i=1}^n W_{n,i}^2 = O_p(1)$. Thus, given that $\max_{i \leq n}W_{n,i} = o_p(1)$, we have $|r_n| \leq o_p(1) \sum_{i=1}^n W_{n,i}^2 = o_p(1)O_p(1) = o_p(1)$ as $n \to \infty$. It remains to show that the assumption holds, namely that $\max_{i \leq n}|W_{n,i}| \overset{p}\to 0$ as $n \to \infty$. Write 
\begin{equation}
W_{n,i}^2 = n^{-1}(h^{\tr}u_{\theta,i})^2
+ \{ W_{n,i}^2 - n^{-1}(h^{\tr}u_{\theta,i})^2\}
= n^{-1}(h^{\tr}u_{\theta,i})^2 + \eta_{n,i} \leq n^{-1}(h^\top u_{\theta,i})^2 + |\eta_{n,i}|.
\notag
\end{equation}
By the triangle inequality and by Markov{'}s inequality, we find that for each $\eps >0$
\begin{align*}
\sum_{i=1}^n \prob_{\theta}( |W_{n,i}| \geq \sqrt{2}\eps \given \Falg_{i-1}) 
& = \sum_{i=1}^n \prob_{\theta}( W_{n,i}^2 \geq 2\eps^2 \given \Falg_{i-1})
 \\
& \leq \sum_{i=1}^n \prob_{\theta}( n^{-1}(h^{\tr}u_{\theta,i})^2
+ | \eta_{n,i}| \geq 2\eps^2  \given \Falg_{i-1})\\
& \leq \sum_{i=1}^n \prob_{\theta}( (h^{\top} u_{\theta,i})^2 \geq n\eps^2 \given \Falg_{i-1}) + 
\sum_{i=1}^n \prob_{\theta}( | \eta_{n,i}| \geq \eps^2 \given \Falg_{i-1}  ) \\
& \leq \frac{1}{n \eps^2}\sum_{i=1}^n \E_{\theta}\,[(h^{\top}u_{\theta,i})^2I\{(h^{\top}u_{\theta,i})^2 \geq n\eps^2\}\given \Falg_{i-1}]
+ \frac{1}{\eps^2} \sum_{i=1}^n \E_{\theta}\,[ |\eta_{n,i}|\given \Falg_{i-1}]\\
& \leq \frac{\vbb{h}^2}{n\eps^2}\sum_{i=1}^n \E_{\theta}\,[\vbb{u_{\theta,i}}^2I\{\vbb{u_{\theta,i}} \geq \sqrt{n}\eps/\vbb{h}\}\given \Falg_{i-1}] + \frac{2}{\eps} \sum_{i=1}^n \E_{\theta}\,[ |\eta_{n,i}|\given \Falg_{i-1}].
\end{align*}
Here we recognise the Lindeberg condition, which says that $n^{-1}\sum_{i=1}^n \E_{\theta}\,[\vbb{u_{\theta,i}}^2I\{\vbb{u_{\theta,i}} \geq \sqrt{n}\eps\}\given \Falg_{i-1}] = o_p(1)$ for each $\eps >0$; and we saw above that $\sum_{i=1}^n \E_{\theta}\,[ |\eta_{n,i}|\given \Falg_{i-1}] = o_p(1)$. From the display above, this shows that $\sum_{i=1}^n \prob_{\theta}( |W_{n,i}| \geq \sqrt{2}\eps \given \Falg_{i-1}) = o_p(1)$ for each $\eps>0$, which by Lemma~2.2 in~\citet[p.~81]{helland1982central} is equivalent to $\max_{i \leq n}|W_{n,i}| = o_p(1)$ as $n$ tends to infinity. This concludes the proof.
\end{proof}

\subsection*{Proof of Lemma~\ref{lemma:irreducible}}
\begin{proof}
    The first step of the proof is to show that for a sufficiently large number $i$, the distribution of $X_i$ emits a density (absolutely continuous with respect to the Lebesgue measure) which is nonzero on the interval $[1/2^i, 1 - 1/2^i]$. Assume first that $X_1 = 1/2 = \xi$. Note that the distribution of $X_2$ emits a density which is nonzero on the set
    $$\mathcal{I}_2 = \left[1/4, 1/4 + \eps\right]\cup\left[3/4 - \eps, 3/4\right],$$
    and that the distribution of $X_3$ emits a density which is nonzero on the set
    $$\mathcal{I}_3 = \left[1/8, 1/8 + 3\eps/2\right]\cup\left[3/8 - \eps/2, 3/8 + \eps\right]\cup\left[5/8 - \eps, 5/8 + \eps/2\right]\cup\left[7/8 - 3\eps/2, 7/8\right].$$
    In general, the distribution of $X_i$ emits a density which is nonzero on the set $\mathcal{I}_i$, where $\mathcal{I}_i$ can be written as the union of $2^{i-1}$ closed intervals, each of length $\mathcal{L}_i$, where
    $$\mathcal{L}_i = \left\{1 + 1/2 + \cdots + 1/{2^{i-2}}\right\}\eps = 2\left\{1 - 1/{2^{i-1}}\right\}\eps.$$
    We will show that for a sufficiently large value of $i$, all the intervals in $\mathcal{I}_i$ overlap, so that $\mathcal{I}_i = [1/2^i, 1 - 1/2^i]$.

    Label the intervals in $\interval_i$ from left to right as $\interval_{i,1}, \dots, \interval_{i,2^{i-1}}$. Note that if we set $\eps=0$, then we would simply recover the dyadic rationals $\mathcal{I}_i = \dyadic_i$. This induces a natural map 
        $$\phi : \dyadic_i \to \{\interval_{i, j} : 1\leq j \leq i\}$$
    mapping $1/2^i \mapsto [1/2^i, 1/2^i + 2\{1 - 1/2^{i-1}\}\eps]$, etc.

    Now, consider only the dyadic rationals $d\in\dyadic_i$ with $d < 1/2$. Note that $\phi(d)$ extends further to the right of $d$ than to the left of $d$. That is, if $\phi(d) = [p, q]$, then $q - d \geq d - p$. Indeed, we see that any such interval $\phi(d)$ must have a term of $+\eps$ added to its right most endpoint. We see this explicitly in the expressions for $\interval_2$ and $\interval_3$ above. Since each interval is of length $\mathcal{L}_i = 2\left\{1 - 1/2^{i-1}\right\}\eps$, this means that all the intervals $\{\phi(d) : d < 1/2\}$ will overlap provided we have 
    \begin{equation}\label{eq:condition_overlap}
        \left\{1 - 1/{2^{i-1}}\right\}\eps > 1/{2^{i-1}}.
    \end{equation}
    Indeed, the left hand side is half the length of each interval $\phi(d)$, and the right hand side is the distance between two consecutive elements in $\dyadic_i$. Setting
        $$N > \log_2(1 + 1/\eps) + 1$$
    ensures that \eqref{eq:condition_overlap} holds for all $i\geq N$. By the same argument, the intervals on the right hand side of $1/2$ will also overlap for $i \geq N$. Since the two middle interval 
    $$\phi\left(\frac{2^{i-2} - 1}{2^{i-1}}\right)\quad\text{and}\quad\phi\left(\frac{2^{i-2} + 1}{2^{i-1}}\right)$$
    will also overlap for such $i$, we conclude that all the interval overlap for $i\geq N$. Hence, for large $i$, the distribution of $X_i$ emits a density which is nonzero on the interval $\interval_i = [1/2^i, 1 - 1/2^i]$.

    The same argument can be used for any starting point $X_1\in[0,1]$, essentially only scaling the \eqref{eq:condition_overlap} by a universally bounded constant. Thus, at any point $x\in[0,1]$, the $n$-step kernel $k^n(x, \cdot)$ will emit a density which is nonzero on $[a_n, b_n]$, where $a_n \to 0 $ and $b_n\to 1$ as $n\to\infty$. This proves that the chain must be irreducible with respect to the Lebesgue measure.
\end{proof}

\subsection*{Proof of Lemma~\ref{lemma:drift_function}}

\begin{proof}
    We have 
        $$\int V(y) k(x, \diff y) - V(x) + 1 = \frac1\eps H_x\int_{\frac{x}2}^{\frac{x}2 + \eps} (my+1)\,\diff y \\ + \frac1\eps[1 - H_x]\int_{\frac{x+1}2-\eps}^{\frac{x+1}2} (my+1)\,\diff y - (mx+1) + 1.
        $$
    Now, letting $x=1$ and simplifying, we have
    \begin{equation}\label{eq:driftx1}
        \int V(y) k(1, \diff y) - V(1) + 1 = m\left[\eps\left( H_1 - 1/2\right) - H_1/2\right] + 1.
    \end{equation}
    In the Markovian Langlie design (with $a=0$ and $b=1$ without loss of generality), it is assumed that $0 < -\alpha/\beta < 1$, and so in particular, $H_1 > 1/2$. Thus, we can choose $\eps>0$ such that
    $$\eps < \frac{H_1}{2(H_1 - 1/2)}$$
    which makes the square bracket in \eqref{eq:driftx1} negative. We can then choose
    $$m > \left[\eps\left(H_1 - 1/2\right) - H_1/2\right]^{-1},$$
    which shows that $\int V(y) k(1, \diff y) - V(1) + 1 < 0$. By continuity of the kernel, this must also hold in an open neighbourhood of $1$.    
\end{proof}

\bibliographystyle{plainnat}
\bibliography{bibliography}

\end{document}